\documentclass[11pt,reqno]{amsart}

\usepackage{pb-diagram, graphics}
\usepackage{setspace}
\usepackage{amscd,amsfonts}
\usepackage{amssymb}
\usepackage{amsmath}
\usepackage{amsthm}
\theoremstyle{plain}
\newtheorem{thm}{Theorem}[subsection]
\newtheorem{cor}[thm]{Corollary}
\newtheorem{lem}[thm]{Lemma}
\newtheorem{prop}[thm]{Proposition}

\theoremstyle{definition}
\newtheorem{defn}[thm]{Definition}

\newtheorem{rem}[thm]{Remark}

\newcommand{\B}{\mathcal{B}}

\newcommand{\0}{\bar 0}

\numberwithin{equation}{subsection}

\def\Z{{\mathbb Z}}

\def\bbc{\mathbb C}
\def\:{\colon}

\newcommand{\fg}{\mathfrak{g}}

\def \fg{\mathfrak{g}}

\def \fh{\mathfrak{h}}

\def \fh{\mathfrak{h}}

\def\C{{\mathbb C}}
\def\Z{{\mathbb Z}}
\def\bbz{{\mathbb Z}}
\def\bbf{\mathbb{F}}

\def\bu{\textbf{U}}
\def\f{\mathcal{F}}
\newcommand{\supp}{\operatorname{supp}}
\newcommand{\lie}[1]{\mathfrak{#1}}\def\span{\textnormal{span}}
\def\m{\mathcal{M}}
\def\bb{\textbf{B}}

\def\cs{\mathcal{S}}
\def\cp{\mathcal{P}}
\begin{document}

\title[Integral bases for the universal enveloping algebras of map superalgebras]
{Integral bases for the universal enveloping algebras of map superalgebras}

\author{Irfan Bagci}
\address{Department of Mathematics \\
University of North Georgia \\
Oakwood, GA 30566}
\email{irfan.bagci@ung.edu}

\author{Samuel Chamberlin}
\address{Computer Science and Mathematics Department\\
Park University\\
Parkville, MO 64152.}
\email{samuel.chamberlin@park.edu}

\begin{abstract}
Let $\fg$ be a finite dimensional complex simple classical Lie superalgebra and $A$ be a commutative, associative algebra with unity over $\bbc$. In this paper we define an integral form for the universal enveloping algebra of the map superalgebra $\fg\otimes A$, and exhibit an explicit integral basis for this integral form.
\end{abstract}
\maketitle
\section{Introduction}

\noindent In 1955, Chevalley investigated integral forms for the classical Lie algebras. This led to the construction of the classical Chevalley groups. Integral forms for the universal enveloping algebras of the classical Lie algebras are necessary to understand the representation theory of these groups. Cartier and Kostant independently found these integral forms in 1966, \cite{K}. The best way to work with these integral forms is via their integral bases (a $\Z$--basis for the integral form). In order to construct these integral bases it is necessary to write particular elements in Poincar\'{e}-Birkhoff-Witt (PBW) order. This is done via straightening identities in the universal enveloping algebra. Once Cartier and Kostant obtained their integral forms it was possible to study Lie groups and Lie algebras over a field of positive characteristic. This generalized Chevalley's groups, and led to representation theory over a field of positive characteristic, \cite{H}.

Serre showed that using only the Cartan matrix one can present a classical Lie algebra by generators and relations. This led to the development of Kac-Moody Lie algebras, which arise using generalized Cartan matrices and Serre's presentation. The simple affine Lie algebras are of this type and are structurally similar to the classical Lie algebras. These (untwisted) simple affine Lie algebras are central extensions of loop algebras. In 1978, Garland extended the theory of integral forms to loop algebras and then to these affine Lie algebras. Using a Chevalley--type bases for these affine Lie algebras he gave integral bases for these integral forms. The required straightening identities in the universal enveloping algebras were considerably more complicated. In 1983, Mitzman extended these results to all simple affine Lie algebras, \cite{M}. Integral forms for the twisted affine Lie algebras were further explored by Fisher-Vasta, \cite{FV}. Integral forms for the quantized universal enveloping algebra associated to a simple finite-dimensional Lie algebra given by Lusztig, \cite{L}. Beck, Chari and Pressley developed integral forms for the quantized universal enveloping algebra associated to an affine Lie algebra, \cite{BCP,CP}.

In 2001, Chevalley bases for the simple classical Lie superalgebras (excluding types $P(n)$ and $Q(n)$) were constructed by Iohara and Koga, \cite{IK}. For type $Q(n)$ Chevalley bases were first given by Brundan and Kleshchev, \cite{BK}. Chevalley bases for type $P(n)$ and a unified treatment of these bases was given by Fioresi and Gavarini in 2012, \cite{FG}. Integral forms and integral bases for the universal enveloping algebras of classical Lie superalgebras first appeared in \cite{IK}. Later they were given by Shu and Wang, \cite{SW}. Again a unified treatment of the subject was given by Fioresi and Gavarini, \cite{FG}. Integral bases for the general linear Lie superalgebra $\lie{gl}(m,n)$ were given by Brundan and Kujawa, \cite{BKu}.

Recently there has been a great deal of interest in map (super)algebras and their representations, \cite{CFK, NSS, Sav}. Given a Lie (super)algebra $\fg$ and any commutative associative complex algebra $A$ the associated map (super)algebra is the Lie (super)algebra $\fg\otimes_\bbc A$ with bracket given by linearly extending the following bracket
$$\left[z\otimes a,z'\otimes a'\right]:=[z,z']_{\fg}\otimes aa',\ z,z'\in\fg,\ a,a'\in A$$

Since the loop (super)algebras are simply the map (super)algebras for which $A=\C[t,t^{-1}]$ it is natural to generalize Garland's work on integral forms and integral bases for the classical loop algebras to the map (super)algebras $\fg\otimes A$, where $\fg$ is a simple classical Lie (super)algebra. Suitable integral forms and integral bases for the classical map algebras were recently obtained by Chamberlin, \cite{C}. The aim of this paper is to give integral forms and integral bases for the map superalgebras $\fg\otimes A$, where $\fg$ is a simple classical Lie superalgebra. We have done this via straightening identities in the universal enveloping algebras. Many of these identities were previously unknown.

In 2007, Jakeli\'{c} and Moura used Garland and Mitzman's work on integral forms to study representations of affine Lie algebras over a field of positive characteristic, \cite{JM}. One application of this work will be to study representations for the map superalgebras $\fg\otimes A$, where $\fg$ is a simple classical Lie superalgebra, over a field of positive characteristic. Another application, which we plan to do in a future work, is to use our straightening identities to illuminate the structure of Weyl modules for the map (super)algebras $\fg\otimes A$ where, $\fg$ is a simple classical Lie (super)algebra.

Our paper is organized as follows:  In Section 2 we fix some notation and review briefly basic facts about classical Lie superalgebras , map superalgebras and record the properties we are going to need in the rest of the paper. Then in Section 3 we state the main theorem of the paper and give some important corollaries.  Next in Section 4 we state and prove all of the necessary straightening identities. In Section 5 we prove the main result of the paper and give a triangular decomposition of our integral form. Finally, in Section 6 we give an example of our integral forms and bases.

\textbf{Acknowledgement}: We would like to thank Vyjayanthi Chari for pointing out \cite{BCP,CP,L}. We would also  like to thank the anonymous referee for useful comments and corrections.

\section{Notation and Preliminaries}

\subsection{}

If $\mathcal{A}$ is an algebra, over a field $\bbf$ of characteristic 0, define an integral form $\mathcal{A}_\bbz$ of $\mathcal{A}$ to be a $\bbz$-algebra such that $\mathcal{A}_\bbz\otimes_\bbz\bbf=\mathcal{A}$. An integral basis for $\mathcal{A}$ is a $\bbz$-basis for $\mathcal{A}_\bbz$. The following notation will be used throughout this manuscript: $\C$ is the set of complex numbers, $\Z_{\geq0}$ is the set of non-negative integers, and $\Z_{>0}$ is the set of positive integers. All vector spaces and algebras we consider will be over the ground field $\C$. A Lie superalgebra is a finite dimensional $\Z_2$-graded vector space $\fg=\fg_{0}\oplus \fg_{{1}}$ with a bracket $[ , ] : \fg\otimes \fg \rightarrow  \fg$ which preserves the $\Z_2$-grading
and satisfies graded versions of the operations used to define Lie algebras.   The even part  $\fg_{0}$ is a Lie algebra under the bracket.

Given any Lie superalgebra $\fg$ let $\bu(\fg)$ be the universal enveloping superalgebra of $\fg$.  $\bu(\fg)$ admits a PBW type basis and if   $x_1, \cdots, x_m$ is  a basis of $\fg_{0}$ and $y_1, \dots, y_n$ is a basis of $\fg_{1}$, then the elements
$$x_1^{i_1} \dots x_m^{i_m}y_1^{j_1}\dots y_n^{j_n}\ \  \text{with} \ \ i_1, \dots, i_m \geq 0  \ \ \text{and} \ \ j_1, \dots, j_n \in \{0, 1\}$$
form a basis of the universal enveloping algebra $\bu(\fg)$. Given $u\in\bu(\fg)$ and $r\in\Z_{\geq0}$ define
$$u^{(r)}:=\frac{u^r}{r!}\textnormal{ and }\binom{u}{r}:=\frac{u(u-1)\cdots(u-r+1)}{r!}.$$
Define $T^0(\fg):=\C$, and for all $j\geq1$, define $T^j(\fg):=\fg^{\otimes j}$, $T(\fg):=\bigoplus_{j=0}^\infty T^j(\fg)$, and $T_j(\fg):=\bigoplus_{k=0}^jT^k(\fg)$. Then set $\bu_j(\fg)$ to be the image of $T_j(\fg))$ under the canonical surjection $T(\fg)\to\bu(\fg)$. Then for any $u\in\bu(\fg)$ \emph{define the degree of $u$} by $$\deg u:=\min_{j}\{u\in\bu_j(\fg)\}$$

\subsection{}
In 1977 V. Kac provided a complete classification of simple Lie superalgebras (cf. \cite{Kac}). The simple finite-dimensional Lie superalgebras are divided into two types based on their even part: they are either classical (when $\fg_{0}$ is reductive) or of Cartan type (otherwise).

Simple classical Lie superalgebras over $\C$ are either isomorphic to a simple Lie algebra or to one of the following Lie superalgebras:
$$A(m,n), m\geq n\geq 0, m+ n\geq 0; \  B(m,n), m\geq 0, n\geq 1 ; \  C(n), n\geq  3;$$
$$D(m,n), m\geq 2, n\geq 1; \  \ P(n), n\geq 2; \  \  Q(n), n\geq 2; $$
$$ F(4); \  G(3); \  \  \text{and} \ \ D(2,1;a), a \in \C - \{0, -1\}.$$

Fix $\fg=\fg_0\oplus\fg_1$ a classical Lie superalgebra. A Cartan subalgebra $\fh$ of $\fg$ coincides with the Cartan subalgebra of $\fg_{0}$.  Fix a Cartan subalgebra $\fh$ of $\fg_{0}$.  A Cartan subalgebra of $\fg$ is diagonalizable. Therefore we have a root decomposition
$$\fg= \bigoplus_{\alpha \in \fh^{\ast}}\fg_{\alpha},$$
where $\fg_{\alpha}:=\{ x \in \fg \mid [h,x] = \alpha(h)x  \ \ \text{for all } \ \  h \in \fh\}$. The set $R :=\{x \in \fg \mid \fg_{\alpha} \neq 0\}$ is called the root system. The $\Z_2$-grading of $\fg$ determines a decomposition of $R$ into the disjoint union of the even roots $R_0$  and the odd roots $R_1$, where $R_0$ is the root system of $\fg_0$ and $R_1$ is the system of weights of the representation of $\fg_0$ in $\fg_1$.
We record some  properties of  the root system of $\fg$ in the following Proposotion for the rest of the paper. For further  information and details we refer the reader to \cite{BK, Kac, Ser}.

\begin{prop}\label{roots} \cite{BK, Kac, Ser}
Let $\fg$ be a classical Lie superalgebra and let  $\fg= \oplus_{\alpha \in \fh^{\ast}}\fg_{\alpha}$ be its root decomposition relative to $\fh$.
\begin{itemize}
\item[(a)] If $\fg \not\in  \{A(1, 1), P(3), Q(n)\}$ then
\begin{itemize}
\item[(i)] $\dim \fg_{\alpha} =1 $ for each $\alpha \in R$.
\item[(ii)]$[{\fg}_{\alpha}, \fg_{\beta}] \neq 0$ if and only if  $\alpha, \beta , \alpha + \beta \in R \cup \{0\}$.
\item[(iii)] If $\alpha$ is in $R$ (or $R_0$,  or $R_1$), then so is $-\alpha$.
\end{itemize}
\item[(b)] $c\alpha \in R$ for $\alpha \neq 0, c \neq \pm 1$ if and only if $\alpha \in R_1, c= \pm 2$.
\item[(c)] If $\fg =A(1,1)$, then $\dim \fg_{\alpha} = 2$ for  $\alpha \in R_1$ and  $\dim \fg_{\alpha} =1 $ for $\alpha \in R_0$.
\item[(d)] If $\fg = Q(n)$, then $\dim (\fg_{\alpha} \cap \fg_0)= dim (\fg_{\alpha} \cap \fg_1) =2$.
\end{itemize}
\end{prop}

In order to make the treatment uniform for the remainder of this work $\fg\notin\{A(1, 1), P(3), Q(n)\}$. These types of Lie superalgebras have reindexed Chevalley bases, which take into account the greater dimensions of the odd root spaces given in $(c)$ and $(d)$ above and for $\fg=P(3)$. In these cases the statement of the Theorem \ref{thm} is the almost the same. We only need to change the index set for the roots, \cite{BK, FG, IK}.

\subsection{Notation}

Throughout rest of this work, unless otherwise noted, $\lie{g}$ will denote a fixed classical Lie superalgebra not of type $A(1,1)$, $P(3)$, nor $Q(n)$ with bracket $[\ ,\ ]_{\lie{g}}$. Fix a distinguished simple root system for $\fg$, $\Delta:=\{\alpha_1,\ldots,\alpha_l\}$ as in \cite{Kac}. Denote by  $R^+$, and $R^-$,  positive roots, negative roots respectively. For $j\in\{0,1\}$ define $R_j^\pm:=R^\pm\cap R_j$. Define $I:=\{1,\ldots,l\}$. Fix a Chevalley basis for $\fg$, $\{h_i\}_{i\in I}\cup\{x_{\alpha}\}_{\alpha\in R}$.

Fix a commutative associative unitary algebra $A$ over $\C$. $\bb$ will denote a fixed $\C$--basis of $A$. The \emph{map superalgebra} of $\lie{g}$ is the $\Z_2$-graded vector space  $\lie{g}\otimes A$, where   $(\lie{g}\otimes A)_0 := \lie{g}_0\otimes A $ and  $(\lie{g}\otimes A)_1:= \lie{g}_1\otimes A $, with bracket given by extending the bracket
$$[z\otimes a, z'\otimes b]:=[z,z']_{\lie{g}}\otimes ab,\ z,z'\in\fg,\ a,b\in A.$$
by linearity. $\lie{g}$ can be embedded in this Lie superalgebra as $\lie{g}\otimes 1$.

If $A$ is the coordinate ring of the scheme $X$ then the Lie superalgebra $\lie{g}\otimes A$ is the Lie superalgebra of regular functions on $X$ with values in $\lie{g}$ with point-wise bracket, \cite{Sav}.

Write the $\C$ basis of the Lie algebra $\lie{sl}_2$ as $\{x^-,h,x^+\}$.  For each $\alpha\in R$, let $\Omega_\alpha:\bu(h\otimes A)\to\bu(\fh\otimes A)$ be the algebra homomorphism defined by
$$h\otimes a\mapsto h_\alpha\otimes a$$

\subsection{Multisets and $p(\chi)$}

Given any set $S$ define a \emph{multiset of elements of $S$} to be a multiplicity function $\chi:S\to\Z_{\geq0}$. Given $\chi\in\f(S)$ define $\supp\chi:=\{s\in S: \chi(s)>0\}$. Define $\f(S):=\{\chi:S\to\Z_{\geq0}:|\supp\chi|<\infty\}$. For $\chi\in\f(S)$ define $|\chi|:=\sum_{s\in S}\chi(s)$. Notice that $\f(S)$ is an abelian monoid under function addition. Define a partial order on $\f(S)$ so that for
$\psi,\chi\in\f(S)$, $\psi\leq\chi$ if $\psi(s)\leq\chi(s)$ for all $s\in S$. Define $\f_k(S):=\{\chi\in\f(S):|\chi|=k\}$ and given $\chi\in\f(S)$ define $\f(S)(\chi):=\{\psi\in\f(S):\psi\leq\chi\}$ and $\f_k(S)(\chi):=\{\psi\in\f(S)(\chi):|\psi|=k\}$. In the case $S=A$ the $S$ will be omitted from the notation. So that $\f:=\f(A)$, $\f_k:=\f_k(A)$, $\f(\chi):=\f(A)(\chi)$ and $\f_k(\chi):=\f_k(A)(\chi)$.

If $\psi\in\f(\chi)$ we define $\chi-\psi$ by standard function subtraction. Also define functions $\pi:\f-\{0\}\to A$ by
$$\pi(\psi):=\prod_{a\in A}a^{\psi(a)}$$
and $\pi(0)=1$, and $\m:\f\to\Z$ by
$$\m(\psi):=\frac{|\psi|!}{\prod_{a\in A}\psi(a)!}$$
For all $\psi\in\f$, $\m(\psi)\in\Z$ because if $\supp\psi=\{a_1,\ldots,a_k\}$ then $\m(\psi)$ is the multinomial coefficient
$$\binom{|\psi|}{\psi(a_1),\ldots,\psi(a_k)}$$

For $s\in S$ define $\chi_s$ to be the characteristic function of the set $\{s\}$. Then for all $\chi\in\f(S)$
$$\chi:=\sum_{s\in S}\chi(s)\chi_s$$

Given $\alpha\in R$ and $S\subset A$ define $X_\alpha:\f(S)\to\bu(\fg\otimes A)$ by
$$X_\alpha(\chi):=\prod_{a\in\supp\chi}\left(x_\alpha\otimes a\right)^{(\chi(a))}$$


Given $\chi\in\f$, recursively define functions $p:\f\to\bu(h\otimes A)$ by $p(0):=1$ and for $\chi\in\f-\{0\}$,
\begin{eqnarray*}
p(\chi)&:=&-\frac{1}{|\chi|}\sum_{\psi\in\f(\chi)-\{0\}}\m(\psi)\left(h\otimes\pi(\psi)\right)p(\chi-\psi)
\end{eqnarray*}

For all $\alpha\in R$, define $p_\alpha(\chi):=\Omega_\alpha(p(\chi))$ and $p_i(\chi):=p_{\alpha_i}(\chi)$.

\begin{rem}
\begin{enumerate}
\item[(1)] The $p_\alpha(\chi)$ are a generalization of Garland's $\Lambda_k(H_\alpha(r))$ because
$p_{-\alpha}\left(k\chi_{t^r}\right)=\Lambda_{k-1}(H_\alpha(r))$, \cite{Gar} page 502.\\

\item[(2)] Given $\alpha\in R$ and $\psi\in\f$ \\ $X_\alpha\left(|\psi|\chi_1\right)X_{-\alpha}(\psi)\equiv(-1)^{|\psi|}p_\alpha(\psi)\mod(\bu(\fg\otimes A)(x_\alpha\otimes A))$, \cite{C} Lemma 5.4.
\end{enumerate}
\end{rem}

We record some basic properties of  $p_\alpha(\chi) $ for the rest of the paper.
\begin{prop}\label{degp}
Let $\alpha,\beta\in R$, $\chi,\varphi\in\f$, and $a\in A$. Then
\begin{enumerate}
\item  $p_\alpha\left(\chi_a\right)=-\left(h_\alpha\otimes a\right)$\label{palpha1}\\




\item $p_\alpha(\chi)=(-1)^{|\chi|}\prod_{a\in A}(h_\alpha\otimes a)^{(\chi(a))}+ \textnormal{ elements of }\bu(h_\alpha\otimes A)\textnormal{ of degree less than }|\chi|$\label{degpalpha}\\

\item $p_\alpha(\chi)p_\beta(\varphi)=p_\beta(\varphi)p_\alpha(\chi)$\label{palphapbeta}
\end{enumerate}
\end{prop}
\begin{proof}
\eqref{palpha1} can be calculated from the definition. To show \eqref{degpalpha} we proceed by induction on $|\chi|$. If $|\chi|=1$ use \eqref{palpha1}. For the inductive step we have for all $\alpha\in R$ and $\chi\in\f$
\begin{eqnarray*}
-|\chi|p_\alpha(\chi)&=&\sum_{\substack{\psi\in\f(\chi)\\|\psi|>1}}\m(\psi)\left(h_\alpha\otimes\pi(\psi)\right)p_\alpha(\chi-\psi)
+\sum_{c\in\supp\chi}\left(h_\alpha\otimes c\right)p_\alpha(\chi-\chi_c)\\
&=&\sum_{c\in\supp\chi}\left(h_\alpha\otimes c\right)(-1)^{|\chi|-1}\prod_{a\in A}(h_\alpha\otimes a)^{((\chi-\chi_c)(a))}\\
&+&\textnormal{ elements of }\bu(h_\alpha\otimes A)\textnormal{ of degree less than }|\chi|\hskip.1in(\textnormal{by the induction hypothesis})\\
&=&-|\chi|(-1)^{|\chi|}\prod_{a\in A}(h_\alpha\otimes a)^{(\chi(a))}+\textnormal{ elements of }\bu(h_\alpha\otimes A)\textnormal{ of degree less than }|\chi|
\end{eqnarray*}
\eqref{palphapbeta} holds because $p_\alpha(\chi)$ and $p_\beta(\varphi)$ are in the center of $\bu(\fg\otimes A)$.
\end{proof}

\begin{rem}\label{hbasis}
Note that proposition \ref{degp}\eqref{degpalpha} tells us that if $B$ is a $\C$-basis for $A$ then $\left\{p_\alpha(\chi):\chi\in\f(B)\right\}$ is a $\C$-basis for $\bu\left(\left\{h_\alpha\right\}\otimes A\right)$.
\end{rem}

\section{An integral form and integral basis}

In this section we define our integral form and state its integral basis as a theorem. We also give an important corollary to this theorem.

\subsection{}
Assume that $A$ has a basis $\bb$, which is closed under multiplication. Then our integral form is defined as follows
\begin{defn}
Define $\bu_\Z(\fg\otimes A)$, to be the $\Z$-subalgebra of $\bu(\fg\otimes A)$ generated by
$$\left(x_\alpha\otimes b\right)^{(s)},(x_\gamma\otimes c),p_i(\chi):\alpha\in R_0,\ \gamma\in R_1,\ b,c\in\bb,\ s\in\Z_{\geq0},\ i\in I,\ \chi\in\f(\bb).$$
\end{defn}

\begin{rem}
Proposition \ref{degp}\eqref{degpalpha} implies that $\bu_\Z(\fg\otimes A)$ is an integral form for $\bu(\fg\otimes A)$.
\end{rem}

Let
$$M:=\left\{\left(x_\alpha\otimes b\right)^{(s)},(x_\gamma\otimes c),p_i(\chi):\alpha\in R_0,\ \gamma\in R_1,\ b,c\in\bb,\ s\in\Z_{\geq0},\ i\in I,\ \chi\in\f(\bb)\right\}$$

Define a \emph{monomial} in $\bu_\Z(\fg\otimes A)$ to be any finite product of elements of the set $M$. Given a monomial $m$ its \emph{factors} are elements of $M$ appearing in $m$.

The goal of this paper is to prove the following theorem, which is the super-version of Theorem 3.2 in \cite{C}.

\begin{thm}\label{thm}
The $\Z$-superalgebra $\bu_\Z(\fg\otimes A)$ is a free $\Z$-module. Let $\left(\preccurlyeq,R\cup I\right)$ be a total order. Then a $\Z$ basis of $\bu_\Z(\fg\otimes A)$ is given by the set $\mathcal{B}$ of all products (without repetitions) of elements of the set
$$\left\{X_\alpha(\chi_\alpha),p_i(\phi_i),X_\gamma(\psi_\gamma)\ |\ \alpha\in R_0,\ i\in I,\ \gamma\in R_1,\ \phi_i,\chi_\alpha,\psi_\gamma\in\f(\bb),\ \psi_\gamma(\bb)\subset\{0,1\}\right\}$$
taken in the order given by $\left(\preccurlyeq,R\cup I\right)$.
\end{thm}

\begin{rem}
\begin{enumerate}
\item In the case where $2\gamma\in R_0$ for some $\gamma\in R_1$ we will also need a total order $\precsim$ on the basis $\bb$ for $A$ in order to state our integral basis for $\bu_\bbz(\fg\otimes A)$ because in this case $(x_\gamma\otimes a)$ and $(x_\gamma\otimes b)$ do not commute. In this case the products in $\B$ will need to first be taken in the order $\left(\preccurlyeq,R\cup I\right)$ and then in the order given by $(\precsim,\bb)$.\label{bborder}\\

\item In the cases $\fg\in\{A(1,1),P(3),Q(n)\}$ the statement of Theorem \ref{thm} must be changed so that we no longer index by $R$ but instead use a larger index set. Other than the new index set this theorem is the same in these cases, \cite{BK, FG, IK}.
\end{enumerate}
\end{rem}

\begin{cor} We have the following isomorphism of $\Z$-modules
$$\bu_\Z(\fg\otimes A)  \cong  \bu_\Z(\fg_0\otimes A) \otimes_{\Z} \Lambda_{\Z}(\fg_1 \otimes A),$$
where  $\Lambda_{\Z}(\fg_1 \otimes A)$ denotes the exterior $\Z$-algebra over $\fg_1 \otimes A$.
\end{cor}
\begin{proof}
Choose a basis $\mathcal{B}$ of $\bu_\Z(\fg\otimes A)$ as in Theorem \ref{thm} where $\preccurlyeq$ is such that $R_0^-\preccurlyeq I\preccurlyeq R_0^+\preccurlyeq R_1$. Then each element of $\bu_\Z(\fg\otimes A)$ can be written uniquely as a $\Z$-linear combination of products of the form
\begin{equation}
\prod_{\alpha\in R_0^-}X_{\alpha}\left(\psi_\alpha\right)\prod_{i\in I}p_i\left(\varphi_i\right)\prod_{\beta\in R_0^+} X_\beta\left(\phi_\beta\right)\prod_{\gamma\in R_1}X_\gamma\left(\tau_\gamma\right)
\end{equation}
where $\psi_\alpha,\varphi_i,\phi_\beta,\tau_\gamma\in\f(\bb)$ for all $\alpha\in R_0^-$, $i\in I$, $\beta\in R_0^+$, and $\gamma\in R_1$, and $\tau_\gamma(b)\leq1$ for all in $\gamma\in R_1$ and all $b\in\bb$. Applying Theorem \ref{thm} to $\fg_0\otimes A$ we see that the $\Z$-span of all monomials of the form $\displaystyle\prod_{\alpha\in R_0^-}X_{\alpha}\left(\psi_\alpha\right)\prod_{i\in I}p_i\left(\varphi_i\right)\prod_{\beta\in R_0^+}X_\beta\left(\phi_\beta\right)$ is $\bu_\Z(\fg_0\otimes A)$. It is clear that the $\Z$-span of all monomials of the form $\displaystyle\prod_{\gamma\in R_1}X_\gamma\left(\tau_\gamma\right)$ is $\Lambda_{\Z}(\fg_1\otimes A)$.
\end{proof}

The remainder of this paper is devoted to the proof of Theorem \ref{thm}. First we will list and then prove all of the necessary straightening identities in $\bu_\Z(\fg\otimes A)$.

\section{Straightening Identities}

In this section we state all commutation relations and provide a proofs for some of these identities. We have commutation relations involving even generators and those involving even and odd generators. The proofs will follow the lists of the straightening identities. Note that in all of these identities all of the coefficients are in $\Z$. Also, note that in the cases which are not listed the factors commute.

\subsection{Even Generator Straightening Identities}

Given $\chi\in\f$ define
$$\cs(\chi):=\left\{\psi\in\f(\f):\sum_{\phi\in\f}\psi(\phi)\phi\leq\chi\right\}\textnormal{ and }\cs_k(\chi):=\cs(\chi)\cap\f_k(\f)$$
Given $j\geq0$ define
$$\cp(j):=\left\{\lambda\in\f\left(\Z_{\geq0}\right):\sum_{m\in\supp\lambda}\lambda(m)m=j\right\}
\textnormal{ and }\cp_k(j):=\cp(j)\cap\f_k\left(\Z_{\geq0}\right)$$
For all $j,k\in\Z_{\geq0}$, $\alpha\in R_0$ and $c,d\in A$, define $D^{\alpha}_{j,0}(d,c):=\delta_{j,0}$, and, for $k>0$, $D^{\alpha}_{j,k}:A^2\to\bu(\fg\otimes A)$ by
$$D^{\alpha}_{j,k}(d,c):=\sum_{\lambda\in\cp_k(j)}\prod_{m\in\supp\lambda}\left(x_{\alpha}\otimes d^mc\right)^{(\lambda(m))}$$

\begin{rem}\label{D}
The definition of $D_{j,k}^{\pm\alpha}(d,c)$ above is equivalent to the recursive definition of $D_\alpha^\pm\left(j\chi_1,j\chi_d,k\chi_c\right)$ from \cite{C} by Proposition 5.2 in \cite{C}.
\end{rem}


\begin{prop}\label{strteven}
For all $\alpha,\beta,\gamma\in R_0$, $i,j\in I$, $\chi,\varphi\in\f$, $a,b\in A$, $r,s\in\Z_{\geq0}$.
\begin{eqnarray}
p_i(\chi)p_j(\varphi)&=&p_j(\varphi)p_i(\chi)\label{pipj}\\
\left(x_{\beta}\otimes b\right)^{(r)}\left(x_{\beta}\otimes b\right)^{(s)}
&=&\binom{r+s}{s}\left(x_{\beta}\otimes b\right)^{(r+s)}\label{xbxb}\\
\left(x_\alpha\otimes a\right)^{(r)}\left(x_{-\alpha}\otimes b\right)^{(s)}
&=&\sum_{\substack{j,k,m\in\Z_{\geq0}\\j+k+m\leq\min(r,s)}}(-1)^{j+k+m}D^{-\alpha}_{j,s-j-k-m}(ab,b)p_\alpha\left(k\chi_{ab}\right)\nonumber\\
&\times&D^\alpha_{m,r-j-k-m}(ab,a)\label{x+x-}\\
\left(x_\alpha\otimes b\right)^{(r)}p_i(\chi)
&=&\sum_{\psi\in\cs_r(\chi)}p_i\left(\chi-\sum_{\phi\in\f}\psi(\phi)\phi\right)\nonumber\\
&\times&\prod_{\phi\in\f}\left(\binom{\alpha(h_i)+|\phi|-1}{|\phi|}\m(\phi)\left(x_\alpha\otimes b\pi(\phi)\right)\right)^{(\psi(\phi))}
\label{x+rpi}\\
p_i(\chi)\left(x_{-\alpha}\otimes b\right)^{(r)}
&=&\sum_{\psi\in\cs_r(\chi)}\prod_{\phi\in\f}\left(\binom{\alpha(h_i)+|\phi|-1}{|\phi|}
\m(\phi)\left(x_{-\alpha}\otimes b\pi(\phi)\right)\right)^{(\psi(\phi))}\nonumber\\
&\times&p_i\left(\chi-\sum_{\phi\in\f}\psi(\phi)\phi\right)\label{pix-r}\\
\left(x_{\gamma}\otimes a\right)^{(r)}\left(x_{\beta}\otimes b\right)^{(s)}
&=&\sum_{\substack{\psi\in\f\left(\Z_{>0}^2\right)\\ r\geq\sum j\psi(j,k)\\ s\geq\sum k\psi(j,k)}}\varepsilon_\psi
\left(x_{\beta}\otimes b\right)^{\left(s-\sum k\psi(j,k)\right)}\prod_{(j,k)\in\supp\psi}\left(x_{j\gamma+k\beta}\otimes a^jb^k\right)^{(\psi(j,k))}
\nonumber\\
&\times&\left(x_{\gamma}\otimes a\right)^{\left(r-\sum j\psi(j,k)\right)}\label{xaxb}
\end{eqnarray}
Where $\varepsilon_\psi\in\{\pm1\}$ for all $\psi\in\f\left(\Z_{>0}^2\right)$, and $x_{j\gamma+k\beta}=0$ if $j\gamma+k\beta\notin R$.
\end{prop}
\begin{proof}
\eqref{pipj} was proved in Proposition \ref{degp}\eqref{palphapbeta}. \eqref{xbxb} is simple. \eqref{x+x-} follows from applying $\Omega_\alpha$ to Lemma 5.4 in \cite{C}. Lemma 5.4 in \cite{C} applies because, for all $\alpha\in R_0$, $\{x_{-\alpha},h_\alpha,x_\alpha\}$ is an $\lie{sl}_2$-triple. Proposition 5.2 in \cite{C} gives the necessary equivalence of our nonrecursive definition of $D_{j,k}^{\pm\alpha}(d,c)$ and the more general recursive definition in \cite{C}, (see Remark \ref{D}). \eqref{x+rpi} can be proved by induction on $r$. For the base case $(r=1)$ we will prove the more general statement given in Lemma \ref{xdeltapi}. The inductive step follows directly from the base case. The details are left to the reader. The proof of \eqref{pix-r} is similar to the proof of \eqref{x+rpi} and hence is omitted. Identity \eqref{xaxb} is a combined statement of the case by case version proved in Lemma \ref{pmbasislem}.
\end{proof}

\begin{rem}
In the cases $\fg\in\{A(1,1),P(3),Q(n)\}$ these identities are essentially the same. We simply need to replace the roots with the corresponding elements in the new index set for the roots, \cite{BK, FG, IK}.
\end{rem}

The following Lemma is a general statement of the $r=1$ case of \eqref{x+rpi}. Note that $\delta$ can be an even or an odd root.

\begin{lem}\label{xdeltapi}
For all $\delta\in R$, $b\in A$, $i\in I$ and $\chi\in\f$
$$\left(x_\delta\otimes b\right)p_i(\chi)
=\sum_{\psi\in\f(\chi)}\binom{|\psi|-1+\delta(h_i)}{|\psi|}\m(\psi)p_i(\chi-\psi)\left(x_\delta\otimes b\pi(\psi)\right)$$
\end{lem}

\begin{proof}
We will prove Lemma \ref{xdeltapi} by induction on $|\chi|$.

\hskip-.1in$|\chi|\left(x_\delta\otimes b\right)p_i(\chi)$
\begin{eqnarray*}
&=&-\sum_{\psi\in\f(\chi)-\{0\}}\m(\psi)(x_\delta\otimes b)\left(h_i\otimes\pi(\psi)\right)p_i(\chi-\psi)\\
&=&-\sum_{\psi\in\f(\chi)-\{0\}}\m(\psi)\left(h_i\otimes\pi(\psi)\right)(x_\delta\otimes b)p_i(\chi-\psi)
+\delta\left(h_i\right)\sum_{\psi\in\f(\chi)-\{0\}}\m(\psi)(x_\delta\otimes b\pi(\psi))p_i(\chi-\psi)\\
&=&-\sum_{\psi\in\f(\chi)-\{0\}}\m(\psi)\left(h_i\otimes\pi(\psi)\right)\sum_{\psi'\in\f(\chi-\psi)}
\binom{|\psi'|-1+\delta(h_i)}{|\psi'|}\m\left(\psi'\right)p_i\left(\chi-\psi-\psi'\right)\left(x_\delta\otimes b\pi\left(\psi'\right)\right)\\
&+&\delta\left(h_i\right)\sum_{\psi\in\f(\chi)-\{0\}}\m(\psi)\sum_{\psi'\in\f(\chi-\psi)}\binom{|\psi'|-1+\delta(h_i)}{|\psi'|}\m\left(\psi'\right)
p_i\left(\chi-\psi-\psi'\right)\left(x_\delta\otimes b\pi(\psi)\pi\left(\psi'\right)\right)\\
&&\textnormal{(by the induction hypothesis)}\\
&=&-\sum_{\psi'\in\f(\chi)-\{\chi\}}\binom{|\psi'|-1+\delta(h_i)}{|\psi'|}\m\left(\psi'\right)\sum_{\psi\in\f\left(\chi-\psi'\right)-\{0\}}\m(\psi)
\left(h_i\otimes\pi(\psi)\right)p_i\left(\chi-\psi-\psi'\right)\left(x_\delta\otimes b\pi\left(\psi'\right)\right)\\
&+&\delta\left(h_i\right)\sum_{\psi\in\f(\chi)-\{0\}}\m(\psi)\sum_{\substack{\psi_1\in\f(\chi)\\ \psi\leq\psi_1}} \binom{|\psi_1|-|\psi|-1+\delta(h_i)}{|\psi_1|-|\psi|}\m\left(\psi_1-\psi\right)p_i\left(\chi-\psi_1\right)
\left(x_\delta\otimes b\pi\left(\psi_1\right)\right)\\
&=&\sum_{\psi_1\in\f(\chi)}\binom{|\psi_1|-1+\delta(h_i)}{|\psi_1|}\m\left(\psi_1\right)(|\chi|-|\psi_1|)p_i(\chi-\psi_1)
\left(x_\delta\otimes b\pi\left(\psi_1\right)\right)\\
&+&\delta\left(h_i\right)\sum_{\psi_1\in\f(\chi)-\{0\}}\sum_{\psi\in\f(\psi_1)-\{0\}}\m(\psi)\binom{|\psi_1|-|\psi|-1+\delta(h_i)}{|\psi_1|-|\psi|}
\m\left(\psi_1-\psi\right)p_i\left(\chi-\psi_1\right)\left(x_\delta\otimes b\pi\left(\psi_1\right)\right)
\end{eqnarray*}
So for Lemma \ref{xdeltapi} it suffices to show that for all $\psi_1\in\f(\chi)$
\begin{eqnarray}
&&\hskip-.3in|\psi_1|\binom{|\psi_1|-1+\delta(h_i)}{|\psi_1|}\m(\psi_1)\nonumber\\
&=&\delta(h_i)\sum_{\psi\in\f(\psi_1)-\{0\}}\m(\psi)\binom{|\psi_1|-|\psi|-1+\delta(h_i)}{|\psi_1|-|\psi|}\m(\psi_1-\psi)\label{comb}
\end{eqnarray}
This will be proved by induction on $|\psi_1|$. It is easily checked if $|\psi_1|=1$. For $|\psi_1|\geq1$ we have
\begin{eqnarray*}
&&\hskip-.3in\delta(h_i)\sum_{\psi\in\f(\psi_1)-\{0\}}\m(\psi)\binom{|\psi_1|-|\psi|-1+\delta(h_i)}{|\psi_1|-|\psi|}\m(\psi_1-\psi)\\
&=&\delta(h_i)\sum_{\psi\in\f(\psi_1)-\{0\}}\sum_{c\in\supp\psi}\m(\psi-\chi_c)\binom{|\psi_1|-|\psi|-1+\delta(h_i)}{|\psi_1|-|\psi|}
\m(\psi_1-\psi)\\
&=&\delta(h_i)\sum_{c\in\supp\psi_1}\sum_{\psi'\in\f(\psi_1-\chi_c)}\m(\psi')\binom{|\psi_1|-|\psi'|-2+\delta(h_i)}{|\psi_1|-|\psi'|-1}
\m(\psi_1-\psi'-\chi_c)\\
&=&\delta(h_i)\sum_{c\in\supp\psi_1}\sum_{\psi'\in\f(\psi_1-\chi_c)-\{0\}}\m(\psi')\binom{|\psi_1|-|\psi'|-2+\delta(h_i)}{|\psi_1|-|\psi'|-1}
\m(\psi_1-\psi'-\chi_c)\\
&+&\delta(h_i)\sum_{c\in\supp\psi_1}\binom{|\psi_1|-2+\delta(h_i)}{|\psi_1|-1}\m(\psi_1-\chi_c)\\
&=&\sum_{c\in\supp\psi_1}(|\psi_1|-1)\binom{|\psi_1|-2+\delta(h_i)}{|\psi_1|-1}\m(\psi_1-\chi_c)\\
&+&\delta(h_i)\sum_{c\in\supp\psi_1}\binom{|\psi_1|-2+\delta(h_i)}{|\psi_1|-1}\m(\psi_1-\chi_c)\hskip.2in\textnormal{(By the induction hypothesis)}\\
&=&(|\psi_1|-1)\binom{|\psi_1|-2+\delta(h_i)}{|\psi_1|-1}\m(\psi_1)+\delta(h_i)\binom{|\psi_1|-2+\delta(h_i)}{|\psi_1|-1}\m(\psi_1)\\
&=&|\psi_1|\binom{|\psi_1|-1+\delta(h_i)}{|\psi_1|}\m(\psi_1)
\end{eqnarray*}
\end{proof}


Given $\alpha,\beta\in R_0$ define $R_{\alpha,\beta}:=\{i\alpha+j\beta:i,j\in\Z\}\cap R_0$. There are three cases for \eqref{xaxb} corresponding to the type of the 2-dimensional root system $R_{\alpha,\beta}$ ($A_2$, $B_2$, or $G_2$). The following lemma explicitly states those cases.

\begin{lem}\label{pmbasislem}
Let $a,b\in A$, $r,s\in\Z_{\geq0}$, and $\alpha,\beta\in R_0$ be given. Then the following identities hold:

$(1)$ If $ R_{\alpha,\beta}$ is of type $A_2$
\begin{eqnarray*}
\left(x_\alpha\otimes a\right)^{(r)}\left(x_\beta\otimes b\right)^{(s)}
&=&\sum_{k=0}^{\min(r,s)}\varepsilon^k\left(x_\beta\otimes b\right)^{(s-k)}\left(x_{\alpha+\beta}\otimes ab\right)^{(k)}
\left(x_\alpha\otimes a\right)^{(r-k)}
\end{eqnarray*}
where $\varepsilon\in\{1,-1\}$ is given by $[x_\alpha,x_\beta]=\varepsilon x_{\alpha+\beta}$.

$(2)$ If $ R_{\alpha,\beta}$ is of type $B_2$, then
\begin{eqnarray*}
\left(x_\alpha\otimes a\right)^{(r)}\left(x_\beta\otimes b\right)^{(s)}
&=&\sum\varepsilon_{k_1,k_2}\left(x_\beta\otimes b\right)^{(s-k_1-k_2)}\prod_{j=1}^2\left(x_{j\alpha+\beta}\otimes a^jb\right)^{(k_j)}\\
&\times&\left(x_\alpha\otimes a\right)^{(r-k_1-2k_2)}
\end{eqnarray*}
where the sum is over all $k_1,k_2\in\Z_{\geq0}$ such that
$k_1+k_2\leq s$ and $k_1+2k_2\leq r$, and
$\varepsilon_{k_1,k_2}\in\{1,-1\}$, for all $k_1,k_2$.

$(3)$ If $ R_{\alpha,\beta}$ is of type $G_2$, then
\begin{eqnarray*}
\left(x_\alpha\otimes a\right)^{(r)}\left(x_\beta\otimes b\right)^{(s)}
&=&\sum\varepsilon_{k_1,k_2,k_3,k_4}\left(x_\beta\otimes b\right)^{\left(s-\sum_{j=1}^3k_j-2k_4\right)}
\prod_{j=1}^3\left(x_{j\alpha+\beta}\otimes a^jb\right)^{(k_j)}\\
&\times&\left(x_{3\alpha+2\beta}\otimes a^3b^2\right)^{(k_4)}
\left(x_\alpha\otimes a\right)^{\left(r-\sum_{j=1}^3jk_j-3k_4\right)}
\end{eqnarray*}
where the sum is over all $k_1,k_2,k_3,k_4\in\Z_{\geq0}$ such
that $k_1+k_2+k_3+2k_4\leq s$ and  $k_1+2k_2+3k_3+3k_4\leq r$,
and $\varepsilon_{k_1,k_2,k_3,k_4}\in\{1,-1\}$, for all
$k_1,k_2,k_3,k_4$.
\end{lem}

\begin{proof}
We will prove part $(1)$ in detail by induction on $s$. The proofs of the other parts are similar and hence are omitted. If $s=0$ each part is trivially true. If $s=1$ we will prove (1) by induction on $r$. The lemma is easily verified by direct computation in the cases $s=1$ and $r\leq3$. Assume (1) for $s=1$ and $r\geq3$. Then for part (1)
\begin{eqnarray*}
&&\hskip-.4in(r+1)\left(x_\alpha\otimes a\right)^{(r+1)}\left(x_\beta\otimes b\right)\\
&=&\left(x_\alpha\otimes a\right)\left(x_\alpha\otimes a\right)^{(r)}\left(x_\beta\otimes b\right)\\
&=&\sum_{k=0}^{1}\varepsilon^k\left(x_\alpha\otimes a\right)\left(x_\beta\otimes b\right)^{(1-k)}
\left(x_{\alpha+\beta}\otimes ab\right)^{(k)}\left(x_\alpha\otimes a\right)^{(r-k)}\hskip.2in(\textnormal{By the induction hypothesis})\\
&=&\left(x_\alpha\otimes a\right)\left(x_\beta\otimes b\right)\left(x_\alpha\otimes a\right)^{(r)}+
\varepsilon\left(x_\alpha\otimes a\right)\left(x_{\alpha+\beta}\otimes ab\right)\left(x_\alpha\otimes a\right)^{(r-1)}\\
&=&(r+1)\left(x_\beta\otimes b\right)\left(x_\alpha\otimes a\right)^{(r+1)}+
\varepsilon\left(x_{\alpha+\beta}\otimes ab\right)\left(x_\alpha\otimes a\right)^{(r)}
+r\varepsilon\left(x_{\alpha+\beta}\otimes ab\right)\left(x_\alpha\otimes a\right)^{(r)}\\
&=&(r+1)\bigg(\left(x_\beta\otimes b\right)\left(x_\alpha\otimes a\right)^{(r+1)}
+\varepsilon\left(x_{\alpha+\beta}\otimes ab\right)\left(x_\alpha\otimes a\right)^{(r)}\bigg)
\end{eqnarray*}

So (1) is true for $s=1$. Proceed by induction on $s\geq1$. Assume the lemma for some $s\geq1$ then in the $(1)$ case
\begin{eqnarray*}
&&\hskip-.4in(s+1)\left(x_\alpha\otimes a\right)^{(r)}\left(x_\beta\otimes b\right)^{(s+1)}\\
&=&\left(x_\alpha\otimes a\right)^{(r)}\left(x_\beta\otimes b\right)\left(x_\beta\otimes b\right)^{(s)}\\
&=&\left(x_\beta\otimes b\right)\left(x_\alpha\otimes a\right)^{(r)}\left(x_\beta\otimes b\right)^{(s)}+\varepsilon\left(x_{\alpha+\beta}\otimes ab\right)\left(x_\alpha\otimes a\right)^{(r-1)}\left(x_\beta\otimes b\right)^{(s)}\\
&=&\sum_{k=0}^{\min(r,s)}\varepsilon^k(s+1-k)\left(x_\beta\otimes b\right)^{(s+1-k)}\left(x_{\alpha+\beta}\otimes ab\right)^{(k)}\left(x_\alpha\otimes a\right)^{(r-k)}\\
&+&\sum_{k=0}^{\min(r-1,s)}\varepsilon^k(k+1)\left(x_\beta\otimes b\right)^{(s-k)}\left(x_{\alpha+\beta}\otimes ab\right)^{(k+1)}\left(x_\alpha\otimes a\right)^{(r-1-k)}(\textnormal{By the induction hypothesis})
\end{eqnarray*}
\begin{eqnarray*}
&=&\sum_{k=0}^{\min(r,s)}\varepsilon^k(s+1-k)\left(x_\beta\otimes b\right)^{(s+1-k)}\left(x_{\alpha+\beta}\otimes ab\right)^{(k)}\left(x_\alpha\otimes a\right)^{(r-k)}\\
&+&\sum_{k=1}^{\min(r-1,s)+1}\varepsilon^kk\left(x_\beta\otimes b\right)^{(s+1-k)}\left(x_{\alpha+\beta}\otimes ab\right)^{(k)}\left(x_\alpha\otimes a\right)^{(r-k)}\\
&=&(s+1)\sum_{k=0}^{\min(r,s+1)}\varepsilon^k\left(x_\beta\otimes b\right)^{(s+1-k)}\left(x_{\alpha+\beta}\otimes ab\right)^{(k)}\left(x_\alpha\otimes a\right)^{(r-k)}
\end{eqnarray*}
\end{proof}

\subsection{Even and Odd Generator Straightening Identities}

If $\alpha,\beta\in R$ define the root string $\alpha$ through $\beta$ to be
$$\sigma_\alpha^\beta:=\left\{\beta-r_{\alpha,\beta}\alpha,\ldots,\beta+q\alpha\right\}$$
Then define $c_{\alpha,\beta}$ as follows
$$c_{\alpha,\beta}:=\left\{\begin{array}{cc}
                0 & \textnormal{if } \alpha+\beta\notin R\\
                \pm\left(r_{\alpha,\beta}+1\right) & \textnormal{if } \lie g\neq P(n)\textnormal{ for }n\neq3
                \textnormal{ and }(\alpha,\alpha)\neq0\textnormal{ or }(\beta,\beta)\neq0\\
                \pm\left(r_{\alpha,\beta}+2\right) & \textnormal{if } \lie g=P(n)\textnormal{ for }n\neq3
                \textnormal{ and }\alpha=\beta_{j,i},\beta=\alpha_{i,j}\\
                \pm\beta(h_\alpha) & \textnormal{if } (\alpha,\alpha)=0=(\beta,\beta)
                \end{array}\right.$$
Then, for all $\alpha,\beta\in R$ with $\alpha\neq-\beta$, $[x_\alpha,x_\beta]=c_{\alpha,\beta}x_{\alpha+\beta}$ and $c_{\alpha,\beta}\in\Z$.

We first list all of the necessary identities involving both even and odd generators and then give some of the proofs.
\begin{prop}\label{strtodd}
For all $m\in\Z_{\geq0}$, $\alpha\in R_0$, $\gamma,\delta\in R_1$, $a,b\in A$, $i\in I$, and $\chi\in\f$.
\begin{eqnarray}
\left(x_\gamma\otimes a\right)p_i(\chi)
&=&\sum_{\psi\in\f(\chi)}\binom{|\psi|-1+\gamma(h_i)}{|\psi|}\m(\psi)p_i(\chi-\psi)\left(x_\gamma\otimes a\pi(\psi)\right)\label{xgammapi}\\
\left(x_\gamma\otimes a\right)^2
&=&\pm2\left(x_{2\gamma}\otimes a^2\right),\ \textnormal{if }2\gamma\in R\label{xgam^2}\\
\left(x_\gamma\otimes a\right)\left(x_{-\gamma}\otimes b\right)
&=&-\left(x_{-\gamma}\otimes b\right)\left(x_\gamma\otimes a\right)+\left(h_\gamma\otimes ab\right),
\hskip.1in\textnormal{ if }\gamma\in R_1\cap(- R_1)\label{xgamx-gam}\\
\left(x_\gamma\otimes a\right)\left(x_\delta\otimes b\right)
&=&-\left(x_\delta\otimes b\right)\left(x_\gamma\otimes a\right)+c_{\gamma,\delta}\left(x_{\gamma+\delta}\otimes ab\right),
\textnormal{ if }\gamma+\delta\neq0\label{xgammaxdelta}\\
\left(x_{\pm\gamma}\otimes a\right)\left(x_{\mp2\gamma}\otimes b\right)^{(m)}
&=&\left(x_{\mp2\gamma}\otimes b\right)^{(m)}\left(x_{\pm\gamma}\otimes a\right)\mp z_\gamma\gamma\left(h_\gamma\right)
\left(x_{\mp2\gamma}\otimes b\right)^{(m-1)}\left(x_{\mp\gamma}\otimes ab\right),\nonumber\\
&&\textnormal{if }2\gamma\in R_0,\ z_\gamma:=\frac{c_{\gamma,\gamma}}{2}\in\{\pm2\}\label{xgamx-2gam^m}\\
\left(x_\alpha\otimes a\right)^{(m)}\left(x_\gamma\otimes b\right)
&=&\sum_{k=1}^m\left(\prod_{s=1}^k\varepsilon_s\right)\binom{r_{\alpha,\gamma}+k}{k}\left(x_{\gamma+k\alpha}\otimes a^kb\right)
\left(x_\alpha\otimes a\right)^{(m-k)}\nonumber\\
&+&\left(x_\gamma\otimes b\right)\left(x_\alpha\otimes a\right)^{(m)},\nonumber\\
&&\textnormal{if }\alpha\neq\pm2\gamma,\ x_{\gamma+k\alpha}:=0\textnormal{ if }\gamma+k\alpha\notin R,\nonumber\\
&&\varepsilon_s=\pm1\textnormal{ such that } \left[x_\alpha,x_{\gamma+(s-1)\alpha}\right]=\varepsilon_s\left(r_{\alpha,\gamma}+s\right)x_{\gamma+k\alpha}\label{xalphmxgam}
\end{eqnarray}
\end{prop}

\begin{proof}
\eqref{xgammapi} is simply Lemma \ref{xdeltapi}. \eqref{xgam^2}, \eqref{xgamx-gam}, and \eqref{xgammaxdelta} are all easily verified by direct computation using Proposition \ref{roots}. \eqref{xgamx-2gam^m} can be easily proved by induction on $m$. We will prove \eqref{xalphmxgam} by induction on $m$. The base case $m=1$ is trivial. For the inductive step we have
\begin{eqnarray*}
&&\hskip-.2in(m+1)\left(x_\alpha\otimes a\right)^{(m+1)}\left(x_\gamma\otimes c\right)\\
&=&\left(x_\alpha\otimes a\right)\left(\left(x_\gamma\otimes c\right)\left(x_\alpha\otimes a\right)^{(m)}
+\sum_{k=1}^m\left(\prod_{s=1}^k\varepsilon_s\right)\binom{r_{\alpha,\gamma}+k}{k}\left(x_{\gamma+k\alpha}\otimes a^kc\right)
\left(x_\alpha\otimes a\right)^{(m-k)}\right)\\
&&(\textnormal{by the induction hypothesis})\\
&=&(m+1)\left(x_\gamma\otimes c\right)\left(x_\alpha\otimes a\right)^{(m+1)}
+\varepsilon_1(r_{\alpha,\gamma}+1)\left(x_{\gamma+\alpha}\otimes ac\right)\left(x_\alpha\otimes a\right)^{(m)}\\
&+&\sum_{k=1}^m\left(\prod_{s=1}^k\varepsilon_s\right)\binom{r_{\alpha,\gamma}+k}{k}
\left(x_{\gamma+k\alpha}\otimes a^kc\right)(m-k+1)\left(x_\alpha\otimes a\right)^{(m-k+1)}\\
&+&\sum_{k=1}^m\left(\prod_{s=1}^k\varepsilon_s\right)\binom{r_{\alpha,\gamma}+k}{k}\varepsilon_{k+1}(r_{\alpha,\gamma}+k+1)
\left(x_{\gamma+(k+1)\alpha}\otimes a^{k+1}c\right)\left(x_\alpha\otimes a\right)^{(m-k)}\\
&=&(m+1)\left(x_\gamma\otimes c\right)\left(x_\alpha\otimes a\right)^{(m+1)}
+\varepsilon_1(r_{\alpha,\gamma}+1)\left(x_{\gamma+\alpha}\otimes ac\right)\left(x_\alpha\otimes a\right)^{(m)}\\
&+&\sum_{k=1}^m\left(\prod_{s=1}^k\varepsilon_s\right)\binom{r_{\alpha,\gamma}+k}{k}
\left(x_{\gamma+k\alpha}\otimes a^kc\right)(m-k+1)\left(x_\alpha\otimes a\right)^{(m-k+1)}\\
&+&\sum_{k=2}^{m+1}\left(\prod_{s=1}^{k-1}\varepsilon_s\right)\binom{r_{\alpha,\gamma}+k}{k}k\varepsilon_{k}
\left(x_{\gamma+k\alpha}\otimes a^{k}c\right)\left(x_\alpha\otimes a\right)^{(m-k+1)}\\
&=&(m+1)\left(x_\gamma\otimes c\right)\left(x_\alpha\otimes a\right)^{(m+1)}
+\varepsilon_1(r_{\alpha,\gamma}+1)\left(x_{\gamma+\alpha}\otimes ac\right)\left(x_\alpha\otimes a\right)^{(m)}\\
&+&m\varepsilon_1(r_{\alpha,\gamma}+1)\left(x_{\gamma+\alpha}\otimes ac\right)\left(x_\alpha\otimes a\right)^{(m)}\\
&+&(m+1)\sum_{k=2}^m\left(\prod_{s=1}^k\varepsilon_s\right)\binom{r_{\alpha,\gamma}+k}{k}
\left(x_{\gamma+k\alpha}\otimes a^kc\right)\left(x_\alpha\otimes a\right)^{(m-k+1)}\\
&+&(m+1)\left(\prod_{s=1}^{m+1}\varepsilon_s\right)\binom{r_{\alpha,\gamma}+m+1}{m+1}\left(x_{\gamma+(m+1)\alpha}\otimes a^{m+1}c\right)\\
&=&(m+1)\left(x_\gamma\otimes c\right)\left(x_\alpha\otimes a\right)^{(m+1)}
+(m+1)\sum_{k=1}^{m+1}\left(\prod_{s=1}^k\varepsilon_s\right)\binom{r_{\alpha,\gamma}+k}{k}
\left(x_{\gamma+k\alpha}\otimes a^kc\right)\left(x_\alpha\otimes a\right)^{(m-k+1)}
\end{eqnarray*}
\end{proof}

\begin{rem}
\begin{enumerate}
\item In the cases $\fg\in\{A(1,1),P(3)\}$ these identities are essentially the same. We simply need to replace the roots with the corresponding elements in the new index set for the roots, \cite{BK, FG, IK}.\\

\item In the cases $\fg=Q(n)$ we will need to replace the roots with the corresponding elements in the new index set for the roots and add a few new identities. These new identities can shown by straightforward calculations or proved by simple inductions, \cite{BK, FG, IK}.
\end{enumerate}
\end{rem}

\section{Proof of the Main Theorem and A Triangular Decomposition}

In this section we will prove Theorem \ref{thm} and give a triangular decomposition of $\bu_\Z(\fg\otimes A)$. The proof will proceed by induction on the degree of monomials in $\bu_\Z(\fg\otimes A)$ (using the definition of degree in Section 2.1) and the following lemmas.

\subsection{}

\begin{lem}\label{brack}
For all $\alpha,\beta,\gamma\in R_0$ $(\beta\neq-\gamma)$, $\delta,\zeta\in R_1$ $i\in I$, $\chi\in\f(\bb)$, $a,b\in\bb$, and $r,s\in\Z_{\geq0}$.
\begin{itemize}
\item [(1)] $\left[\left(x_\alpha\otimes a\right)^{(r)},\left(x_{-\alpha}\otimes b\right)^{(s)}\right]\in\Z$--span $\B$ and has degree less than $r+s$.\\

\item [(2)] $\left[\left(x_\beta\otimes a\right)^{(r)},p_i(\chi)\right]\in\Z$--span $\B$ and has degree less than $r+|\chi|$.\\

\item [(3)] $\left[\left(x_\beta\otimes a\right)^{(r)},\left(x_\gamma\otimes b\right)^{(s)}\right]$ has degree less than $r+s$.\\

\item [(4)] $\left[\left(x_\delta\otimes a\right),p_i(\chi)\right]$ has degree less than $|\chi|+1$.\\

\item [(5)] $\left[\left(x_\beta\otimes a\right)^{(r)},\left(x_\delta\otimes b\right)\right]$ has degree less than $r+1$.\\

\item [(6)] $\left[\left(x_\delta\otimes a\right),\left(x_\zeta\otimes b\right)\right]$ has degree less than $2$.\\

\item [(7)] $\left[\left(x_{\pm\gamma}\otimes a\right),\left(x_{\mp2\gamma}\otimes b\right)^{(m)}\right]$ has degree less than $m+1$.
\end{itemize}
\end{lem}
\begin{proof}
For $(1)$ by \eqref{x+x-} we have
\begin{eqnarray*}
\left(x_\alpha\otimes a\right)^{(r)}\left(x_{-\alpha}\otimes b\right)^{(s)}
&=&\sum_{\substack{j,k,m\in\Z_{\geq0}\\0<j+k+m\leq\min(r,s)}}(-1)^{j+k+m}D^{-\alpha}_{j,s-j-k-m}(ab,b)p_\alpha\left(k\chi_{ab}\right)
D^\alpha_{m,r-j-k-m}(ab,a)\\
&+&\left(x_{-\alpha}\otimes b\right)^{(s)}\left(x_\alpha\otimes a\right)^{(r)}\hskip.2in(\textnormal{by definition } D^{\beta}_{0,n}(d,c)=\left(x_\beta\otimes c\right)^{(n)})
\end{eqnarray*}
By the definition of $D^{\beta}_{j,k}(d,c)$, $D^{-\alpha}_{j,s-j-k-m}(ab,b)p_\alpha\left(k\chi_{ab}\right)D^\alpha_{m,r-j-k-m}(ab,a)\in\Z$--span $\B$. Either the sum on the right-hand side is zero or its terms have degree $r+s-2j-k-2m<r+s$.

For $(2)$ by \eqref{x+rpi} we have
\begin{eqnarray*}
&&\hskip-.4in\left(x_\alpha\otimes a\right)^{(r)}p_i(\chi)\\
&=&\sum_{\psi\in\cs_r(\chi)-\{r\chi_0\}}p_i\left(\chi-\sum_{\phi\in\f}\psi(\phi)\phi\right)\prod_{\phi\in\f}\left(\binom{\alpha(h_i)+|\phi|-1}{|\phi|}
\m(\phi)\left(x_\alpha\otimes a\pi(\phi)\right)\right)^{(\psi(\phi))}\\
&+&p_i(\chi)\left(x_\alpha\otimes a\right)^{(r)}
\end{eqnarray*}
and by \eqref{pix-r} we have
\begin{eqnarray*}
&&\hskip-.4in p_i(\chi)\left(x_{-\alpha}\otimes a\right)^{(r)}\\
&=&\sum_{\psi\in\cs_r(\chi)-\{r\chi_0\}}\prod_{\phi\in\f}\left(\binom{\alpha(h_i)+|\phi|-1}{|\phi|}
\m(\phi)\left(x_{-\alpha}\otimes a\pi(\phi)\right)\right)^{(\psi(\phi))}p_i\left(\chi-\sum_{\phi\in\f}\psi(\phi)\phi\right)\\
&+&\left(x_{-\alpha}\otimes a\right)^{(r)}p_i(\chi)
\end{eqnarray*}
In both cases the factors are in the order specified  by $(\preccurlyeq,R\cup I)$ and either the sums on the right-hand sides are zero or by Proposition \ref{degp} the sums on the right-hand sides have degree
$$\left|\chi-\sum_{\phi\in\f}\psi(\phi)\phi\right|+|\psi|=|\chi|+r-\sum_{\phi\in\f}\psi(\phi)|\phi|<|\chi|+r$$
since $\psi\neq r\chi_0$.

For $(3)$ by \eqref{xaxb} we have
\begin{eqnarray*}
&&\hskip-.4in\left(x_\beta\otimes a\right)^{(r)}\left(x_{\gamma}\otimes b\right)^{(s)}\\
&=&\sum_{\substack{\psi\in\f\left(\Z_{>0}^2\right)-\{0\}\\ r\geq\sum j\psi(j,k)\\ s\geq\sum k\psi(j,k)}}\varepsilon_\psi
\left(x_{\gamma}\otimes b\right)^{\left(s-\sum k\psi(j,k)\right)}\prod_{(j,k)\in\supp\psi}\left(x_{j\beta+k\gamma}\otimes a^jb^k\right)^{(\psi(j,k))}
\left(x_{\beta}\otimes a\right)^{\left(r-\sum j\psi(j,k)\right)}\\
&+&\left(x_{\gamma}\otimes b\right)^{(s)}\left(x_{\beta}\otimes a\right)^{(r)}
\end{eqnarray*}
Either the sum on the right-hand side is zero or its terms have degree
$$r+s-|\psi|-\sum_{(j,k)\in\supp\psi}\left(j\psi(j,k)+k\psi(j,k)\right)<r+s$$

For (4) by Lemma \ref{xdeltapi} we have
\begin{eqnarray*}
\left(x_\delta\otimes a\right)p_i(\chi)
&=&\sum_{\psi\in\f(\chi)-\{0\}}\binom{|\psi|-1+\delta(h_i)}{|\psi|}\m(\psi)p_i(\chi-\psi)\left(x_\delta\otimes a\pi(\psi)\right)
+p_i(\chi)\left(x_\delta\otimes a\right)
\end{eqnarray*}
Either the sum on the right-hand side is zero or its terms have degree $|\chi|-|\psi|+1<|\chi|+1$.

For (5), (6) and (7) just apply \eqref{xgamx-gam}, \eqref{xgammaxdelta}, \eqref{xgamx-2gam^m}, and \eqref{xalphmxgam} as needed.
\end{proof}

\begin{lem}\label{degpi}
For all $i\in I$ and $\chi,\varphi\in\f(\bb)$
$$p_i(\chi)p_i(\varphi)=\prod_{a\in A}\binom{(\chi+\varphi)(a)}{\chi(a)}p_i(\chi+\varphi)+u$$
where $u$ is in $\Z-\span\left\{p_i(\psi):\psi\in\f(\bb)\right\}$ with $\deg u<|\chi|+|\varphi|$.
\end{lem}
\begin{proof}
Adapting the proof of Lemma 9.2 in \cite{Gar} shows that if $\chi,\varphi\in\f(\bb)$ then $p(\chi)p(\varphi)\in\Z-\span\left\{p(\psi):\psi\in\f(\bb)\right\}$.  Applying $\Omega_{\alpha_i}$ and using Proposition \ref{degp}\eqref{degpalpha} for the degree gives the result.
\end{proof}

\noindent\emph{Proof of Theorem \ref{thm}.} Fix an order $(\preccurlyeq,R\cup I)$. Since $\bb$, $\{p_\alpha(\chi):\chi\in\f(\bb)\}$, and $\{x_\alpha\}_{\alpha\in R}\cup\{h_i\}_{i\in I}$ are $\C$-bases for $A$, $\bu\left(\left\{h_\alpha\right\}\otimes A\right)$, and $\fg$ respectively the Poincar\'{e}-Birkhoff-Witt Theorem for Lie superablgebras implies that $\B$ is a $\C$--linearly independent set. Hence $\B$ is a $\bbz$--linearly independent set.

The proof that the $\Z$-span of $\B$ is $\bu_\Z(\fg\otimes A)$ will proceed by induction on the degree of monomials in $\bu_\Z(\fg\otimes A)$. Since $p_i(\chi_a)=-(h_i\otimes a)$ any degree one monomial is in the $\Z$-span of $\B$. Take any monomial $m$ in $\bu_\Z(\fg\otimes A)$. If $m\in\B$ then we are done. If not then either the factors of $m$ are not in the order specified by $\preccurlyeq$, or $m$ has multiples of factors with the following forms
\begin{eqnarray}
\left(x_\alpha\otimes b\right)^{(r)}\textnormal{ and }\left(x_\alpha\otimes b\right)^{(s)}&,&\alpha\in R_0,\ b\in\bb,\ r,s\in\Z_{\geq0}\label{facxbxb}\\
p_i(\chi)\textnormal{ and }p_i(\psi)&,&i\in I,\ \chi,\psi\in\f(\bb)\label{facpipi}\\
\left(x_\gamma\otimes c\right)^{j}&,&\gamma\in R_1,\ c\in\bb,\ j\in\Z_{\geq0}\label{facxgam}.
\end{eqnarray}
If the factors of $m$ are not in the order given by $\preccurlyeq$ then we can rearrange the factors of $m$ using the straightening identities in Propositions \ref{strteven} and \ref{strtodd}. Once this is done Lemma \ref{brack} guarantees that each rearrangement will only produce $\Z$-linear combinations of monomials in the correct order with lower degree. These lower degree monomials are then in the $\Z$-span of $\B$ by the induction hypothesis.

If (possibly after rearranging factors as above) $m$ contains the products of the pairs of factors in \eqref{facxbxb} or \eqref{facpipi} we apply \eqref{xbxb} or Lemma \ref{degpi} respectively to consolidate these pairs of factors into single factors with integral coefficients. If $m$ contains a power of factors as in \eqref{facxgam} we apply \eqref{xgam^2} to get
$$\left(x_\gamma\otimes c\right)^{j}=\left(x_\gamma\otimes c\right)^{2k+\varepsilon}=z\left(x_{2\gamma}\otimes c^2\right)^{k}\left(x_\gamma\otimes c\right)^{\varepsilon}=zk!\left(x_{2\gamma}\otimes c^2\right)^{(k)}\left(x_\gamma\otimes c\right)^{\varepsilon}$$
where $z\in\Z$, $k\in\Z_{\geq0}$, $\varepsilon\in\{0,1\}$. We might need to rearrange the factors again using Lemmas \ref{brack} and \ref{degpi} after this step. In the end we see that $m\in\Z$--$\span\B$. Thus the $\Z$-span of $\B$ is $\bu_\Z(\fg\otimes A)$ and hence $\B$ is an integral basis for $\bu_\Z(\fg\otimes A)$.

In the case $2\gamma\in R$ for some $\gamma\in R$ we also have a total order $(\precsim,\bb)$ (see remark \ref{bborder}). Using the identity \eqref{xgammaxdelta} with $\delta=\gamma$ we can reorder the pairs of factors of the form $\left(x_\gamma\otimes a\right)\left(x_\gamma\otimes b\right)$ as required by $\precsim$. With each reordering although new monomials are be created they will have lower degree and hence will be in the $\Z$-span of $\B$ by the induction hypothesis. \hskip4.5in $\square$

\subsection{}

Let  $\bu_{\Z}^\pm(\fg \otimes A)$ denote the $\Z$-subalgebras of $\bu_{\Z}(\fg \otimes A)$ generated, respectively, by
$$\{(x_{\alpha}\otimes a)^{(r)},x_\gamma\otimes c:\alpha\in R_0^\pm,\ \gamma\in R_1^\pm,\ a,c\in\bb,\ r\in\Z_{\geq 0}\}.$$
Let $\bu_{\Z}^{0}(\fg \otimes A)$  denote the $\Z$-subalgebra of $\bu_{\Z}(\fg\otimes A)$ generated by
$$\{p_i(\chi):\chi\in\mathcal{F}(\bb),i\in I\}.$$

Then as a corollary to Theorem \ref{thm} we obtain the following triangular decomposition of $\bu_\Z(\fg\otimes A)$.

\begin{cor}
$$\bu_{\Z}(\fg \otimes A)=\bu_{\Z}^{-}(\fg \otimes A)\bu_{\Z}^{0}(\fg \otimes A)\bu_{\Z}^{+}(\fg \otimes A)$$
\end{cor}
\begin{proof}
In Theorem \ref{thm} choose $(\preccurlyeq,R)$ so that $R^-\preccurlyeq I\preccurlyeq R^+$. Then by Theorem \ref{thm} we can write every element of $\bu_{\Z}(\fg \otimes A)$ as a $\Z$-linear combination of elements of $\bu_{\Z}^{-}(\fg\otimes A)\bu_{\Z}^{0}(\fg\otimes A)\bu_\Z^{+}(\fg\otimes A)$.
\end{proof}

Let $\left(\preccurlyeq_\pm,R^\pm\right)$ be any total orders. Define $\B^\pm$ to be the sets of all products (without repetitions) of elements of the set
$$\left\{X_\alpha(\chi_\alpha),X_\gamma(\psi_\gamma)\ |\ \alpha\in R_0^\pm,\ \gamma\in R_1^\pm,\ \chi_\alpha,\psi_\gamma\in\f(\bb),\ \psi_\gamma(\bb)\subset\{0,1\}\right\}$$
taken in the orders given by $\preccurlyeq_\pm$. Given a total order $(\preccurlyeq,I)$ define
$$\B^0:=\left\{\prod_{i\in I}p_i(\varphi_i):\varphi_1,\ldots,\varphi_l\in\f(\Z_{\geq0})\right\}$$
with the products taken in the order given by $\preccurlyeq$.

\begin{prop}
Let  $\mathcal{B}^\pm$ and $\mathcal{B}^0$ be as above. Then
\begin{itemize}
\item [(1)] $\mathcal{B}^\pm$ are $\Z$ bases for $\bu_{\Z}^{\pm}(\fg \otimes A)$ respectively.
\item [(2)] $\mathcal{B}^{0}$ is a $\Z$ basis for $\bu_{\Z}^{0}(\fg \otimes A)$.
\end{itemize}
\end{prop}
\begin{proof}
Similarly to the proof of Theorem \ref{thm}. (1) can be proved by induction on the degree of a monomial in $\bu_{\Z}^{\pm}(\fg \otimes A)$ using Lemma \ref{brack}(3) and (4) and \eqref{xbxb}. For (2) note that it is enough to show that any product of elements of $\{p_i(\chi): \chi \in \mathcal{F}(B), i \in I \}$ is in the $\Z-$span of $\bu_{\Z}^{0}(\fg \otimes A)$. To show this we will proceed by induction on the degree of such a product. For the inductive step since $p_i(\chi) $ and $p_j(\chi')$ commute (for $i\neq j$ in $I$) it will suffice to apply Lemma \ref{degpi}.
\end{proof}

\section{Example}

In this section we give a specific example of the integral form and basis.

Let $\fg=A(1,0)=\mathfrak{sl}(2,1)$ and $A=\bbc[t]$. Then $\fg\otimes A=\mathfrak{sl}(2,1)\otimes A$ is a map superalgebra where the bracket on $\fg$ is the super commutator bracket. That is,
$$[x,y]=xy-(-1)^{\deg x\deg y}yx$$
for homogeneous $x,y$ and extended to all $\fg$ by linearity.

Recall that $\fg=\mathfrak{sl}(2,1)$ consists of $3\times3$ block  matrices of the form
\[\begin{bmatrix}
\begin{array}{c c |c}
a & b & c\\
d & e & f\\\hline
g & h & a+e
\end{array}
\end{bmatrix}\]
where $a,b,c,d,e,f,g,h,\in\C$.

Given a matrix $\left[a_{i,j}\right]$ define $\varepsilon_k\left(\left[a_{i,j}\right]\right):=a_{k,k}-a_{k+1,k+1}$. Define $\alpha_1:=\varepsilon_1-\varepsilon_2$ and $\alpha_2:=\varepsilon_2-\delta_1$, where $\delta_1=\varepsilon_3$. Fix the distinguished simple root system $\Delta=\{\alpha_1,\alpha_2\}$. Then $R_{0}^+=\{\alpha_1\}$ and $R_{1}^+=\{\alpha_2,(\alpha_1+\alpha_2)\}$. Fix the following Chevalley basis for $\lie{sl}(2,1)$: $h_1:=h_{\alpha_1}=e_{1,1}-e_{2,2}$ and, $h_2:=h_{\alpha_2}=e_{2,2}+e_{3,3}$. For the root vectors take $x_1:=x_{\alpha_1}=e_{1,2}$, $x_{-1}:=x_{-\alpha_1}=e_{2,1}$, $x_2:=x_{\alpha_2}=e_{2,3}$, $x_{-2}:=x_{-\alpha_2}=e_{3,2}$, and $x_3:=x_{\alpha_1+\alpha_2}=e_{1,3}$, $x_{-3}:=x_{-(\alpha_1+\alpha_2)}=e_{3,1}$.




Note that $\{x_{-1},h_1,x_1\}$ is an $\mathfrak{sl}_2$-triple.

$\bb=\{t^j\mid j\in\Z_{\geq 0}\}$ is a basis for $A$ which is closed under multiplication. Each $\chi\in\f(\bb)$ corresponds in a natural way to a $\chi'\in\f(\Z_{\geq0})$.

$\bu_\Z(\fg\otimes A)$ is the $\Z$-subalgebra of $\bu(\fg\otimes A)$ generated by
$$\left\{\left(x_{\pm1}\otimes t^m\right)^{(s)},x_{\pm j}\otimes t^k,p_i(\chi'):s,k,m\in\Z_{\geq 0},\ \ j=2,3,\ \ i=1,2,\ \ \chi'\in\f(\Z_{\geq0})\right\}$$
The $\Z$-subalgebras $\bu_\Z^\pm(\fg\otimes A)$ are generated by
$$\left\{\left(x_{\pm1}\otimes t^j\right)^{(r)},x_{\pm m}\otimes t^n:m\in\{2,3\},\ j,n,r\in\Z_{\geq0}\right\}$$
The $\Z$-subalgebra $\bu_\Z^0(\fg\otimes A)$ is generated by
$$\left\{p_1(\chi),p_2(\chi'):\chi,\chi'\in\f(\Z_{\geq0})\right\}$$

Given $\alpha\in R$, $j\in\{\pm1,\pm2,\pm3\}$ and $\chi\in \f(\bbz_{\geq0})$ define
$$X_j(\chi):=\prod_{m=0}^\infty\left(x_j\otimes t^m\right)^{(\chi(m))}$$

If we fix an order $\left(\preccurlyeq,R\cup\{1,2\}\right)$ such that $$-\alpha_1\preccurlyeq-\alpha_2\preccurlyeq-(\alpha_1+\alpha_2)\preccurlyeq1\preccurlyeq2\preccurlyeq\alpha_1\preccurlyeq\alpha_2\preccurlyeq
(\alpha_1+\alpha_2)$$
then the $\Z$-basis $\B$ for $\bu_\Z(\fg\otimes A)$ from Theorem \ref{thm} is
$$\left\{\prod_{j=1}^3X_{-j}(\phi_{-j})\prod_{i=1}^2p_i(\varphi_i)\prod_{j=1}^3X_j(\phi_{j})
:\varphi_1,\varphi_2,\phi_{\pm1},\phi_{\pm2},\phi_{\pm3}\in\f(\Z_{\geq0}),\ \phi_{\pm2}(k)\phi_{\pm3}(k)\leq1,\ \forall k\in\Z_{\geq0}\right\}$$

The $\Z$-bases $\B^\pm$ for $\bu_\Z^\pm(\fg\otimes A)$ are
$$\left\{\prod_{j=1}^3X_{\pm j}(\phi_{\pm j}):\phi_{\pm1},\phi_{\pm2},\phi_{\pm3}\in\f(\Z_{\geq0}),\ \phi_{\pm2}(k)\phi_{\pm3}(k)\leq1,\ \forall k\in\Z_{\geq0}\right\}$$

The $\Z$-basis $\B^0$ for $\bu_\Z^0(\fg\otimes A)$ is $\left\{p_1(\varphi_1)p_2(\varphi_2):\varphi_1,\varphi_2\in\f(\Z_{\geq0})\right\}$.

\end{document}